\documentclass[a4paper,12pt]{article}

\newenvironment{proof}{\noindent {\bf Proof:}}{\hfill $\Box$}
\newtheorem{theorem}{Theorem}[section]
\newtheorem{lemma}{Lemma}[section]

\newtheorem{corollary}{Corollary}[section]

\newtheorem{remark}{Remark}[section]

\usepackage{amsmath}
\usepackage{amssymb}
\usepackage{amsfonts}
\usepackage{graphicx}

\title{\bf A hierarchy of LMI inner approximations of the set of stable polynomials}

\begin{document}

\author{Mustapha Ait Rami$^1$,
Didier Henrion$^{2,3}$}

\footnotetext[1]{
Departamento de Ingenier\'{\i}a de Sistemas y Autom\'atica,
Facultad de Ciencias, Universidad de Valladolid, calle Real de Burgos s/n,
SP-47011 Valladolid, Spain. {\tt aitrami@autom.uva.es}}

\footnotetext[2]{
CNRS; LAAS; 7 avenue du colonel Roche, F-31077 Toulouse, France;
Universit\'e de Toulouse; UPS, INSA, INP, ISAE; LAAS; F-31077 Toulouse, France.
{\tt henrion@laas.fr}}

\footnotetext[3]{
Faculty of Electrical Engineering, Czech Technical University in Prague, Technick\'a 4,
CZ-16626 Prague, Czech Republic.
{\tt henrion@fel.cvut.cz}}

\maketitle

\begin{abstract}
Exploiting spectral properties of symmetric banded Toeplitz matrices, we describe simple sufficient conditions for positivity of a trigonometric polynomial formulated as linear matrix inequalities (LMI) in the coefficients. As an application of these results, we derive a hierarchy of convex LMI inner approximations (affine sections of the cone of positive definite matrices of size $m$) of the nonconvex set of Schur stable polynomials of given degree $n < m$. It is shown that when $m$ tends to infinity the hierarchy converges to a lifted LMI approximation (projection of an LMI set defined in a lifted space of dimension quadratic in $n$) already studied in the technical literature.
\end{abstract}

\begin{center}
\end{center}

\section{Introduction}

Linear system stability can be formulated algebraically in the space
of coefficients of the characteristic polynomial. The region of stability
is generally {\it nonconvex} in this space, and this is a major
obstacle when solving fixed-order or robust controller design problems.
In the case of discrete-time linear systems, the region of stability
is a bounded open set whose boundary consists of (flat) hyperplanes and
nonconvex (negatively curved) algebraic varieties. Recent results on real algebraic
geometry and generalized problems of moments
can be used to build up a hierarchy of convex linear matrix inequality (LMI)
outer approximations of the region of stability, with asymptotic
convergence to its convex hull, see e.g. \cite{csm} for a software
implementation and examples. It is generally more difficult to
construct convex LMI {\it inner approximations}, see \cite{hsk03}
for a survey. Strict positive realness of rational transfer functions
and its connection with polynomial positivity conditions are
used in \cite{hsk03} to generate inner approximations which are
{\it lifted} LMI sets. For polynomials of degree $n$, they are projections
onto coefficient space $\mathbb{R}^n$ of an LMI set living in a
lifted space $\mathbb{R}^{\frac{n^2+3n}{2}}$. The LMI set is built
around a particular point, the central polynomial, whose relevance
in robust control design is explained in \cite{hsk03}. These lifted
LMI regions are also used in signal processing, see e.g.
\cite[Section 7.3]{dumitrescu}. They can be derived in a state-space
setting with the Kalman-Yakubovich-Popov lemma \cite{karimi}.

Whereas lifted LMIs are a powerful modeling paradigm (it is currently
conjectured that every convex semialgebraic set is a lifted LMI set),
the introduction of a large number of lifting variables can be seen
as a drawback. It is therefore relevant to build convex LMI inner
approximations of the nonconvex stability region {\it without
liftings}, namely as affine sections of the cone of positive
semidefinite matrices. This is the objective of this paper.
We use results of functional analysis on sequences of eigenvalues
of Toeplitz matrices to derive sufficient LMI conditions for positivity of
trigonometric polynomials, and we apply these results to construct
a hierarchy of $m$-by-$m$ LMI inner approximations of the nonconvex
stability domain. Moreover we prove that when $m$ tends to infinity,
the hierarchy converges asymptotically to the lifted
LMI approximation of \cite{hsk03}.

\section{Trigonometric polynomials and Toeplitz matrices}

Let $p_k$, $k=0,1,2,\ldots,n$ denote real numbers, and
define the trigonometric polynomial
\[
\begin{array}{r@{\;}c@{\;}l}
z=e^{i\theta} \mapsto p(\theta) & = & p_0+p_1(z+z^{-1})+p_2(z^2+z^{-2})+
\cdots+p_n(z^n+z^{-n}) \\
& = & p_0 + 2p_1\cos\theta + 2p_2\cos2\theta+\cdots+2p_n\cos n\theta
\end{array}
\]
of degree $n$ mapping the unit disk of the complex plane onto the real axis.

For a given integer $m > n$, define the column vector
$v_m(z) = [1 \ z \ z^2 \cdots z^{m-1}]^T$ and represent polynomial $p$
as a quadratic form
\begin{equation}\label{quad}
p(\theta) = \frac{1}{m}v^T_m(e^{-i\theta})P_mv_m(e^{i\theta})
\end{equation}
where
\begin{equation}\label{pm}
P_m = \left[\begin{array}{ccccc}
p_0 & \frac{m}{m-1}p_1 & \frac{m}{m-2}p_2 \\
\frac{m}{m-1}p_1 &  p_0  & \frac{m}{m-1}p_1\\
\frac{m}{m-2}p_2  & \frac{m}{m-1}p_1 & p_0\\
& & & \ddots \\
& & & & p_0
\end{array}\right]
\end{equation}
is an $m$-by-$m$ symmetric banded Toeplitz matrix.

Define
\[
R_m = \left[\begin{array}{ccccc}
p_0 & p_1 & p_2 \\
p_1 & p_0 & p_1 \\
p_2 & p_1 & p_0 \\
& & & \ddots \\
& & & & p_0
\end{array}\right]
\]
as the $m$-by-$m$ moment matrix of $p$, so named for
\[
p_k = \frac{1}{2\pi} \int_0^{2\pi} p(\theta)e^{-ik\theta} d\theta
\]
is the $k$-th moment, or Fourier coefficient, of polynomial $p$.
Note that $R_m$ has the
same banded symmetric Toeplitz structure as $P_m$.

Connections between the spectrum of matrix $R_m$ and
the values taken by polynomial $p$ on the unit circle
have been studied extensively.
In the sequel, $\lambda_{\min}$ denotes the minimum eigenvalue
of a symmetric matrix.

\begin{theorem}\label{Szego} 
$$\lim_{m\rightarrow +\infty}\lambda_{\min}(R_m)=\min_\theta p(\theta).$$
\end{theorem}

\begin{proof}
It is a corollary of G\'abor Szeg\H{o}'s fundamental eigenvalue
distribution theorem, see e.g. \cite[Corollary 4.2]{gray}.
\end{proof}

In this section we aim at establishing a similar spectral property linking
matrix $P_m$ and polynomial $p$. First let us state a few instrumental
results.

\begin{lemma}\label{lem1}
For all $\theta$ it holds $\lambda_{\min}(P_m)\leq p(\theta)$ and as a consequence 
\begin{equation}
\limsup_{m\rightarrow +\infty}\lambda_{\min}(P_m)\leq \min_\theta p(\theta). 
\end{equation}
\end{lemma}

\begin{proof}
From relation (\ref{quad}) and the identity $v^T_m(e^{-i\theta})v_m(e^{i\theta})=m$,
it follows that
\begin{equation}\label{rayleigh}
\frac{v^T_m(e^{-i\theta})P_mv_m(e^{i\theta})}{v^T_m(e^{-i\theta})v_m(e^{i\theta})} =
p(\theta)
\end{equation}
and hence $\lambda_{\min}(P_m)\leq p(\theta)$. When $m \rightarrow \infty$
we obtain the desired result.
\end{proof}

\begin{lemma}\label{lem2}
\[
 \| P_m-R_m\| = O(m^{-\frac{1}{2}}).
\]
\end{lemma}

\begin{proof}
Consider
\[
P_m-R_m = \left[\begin{array}{ccccc}
0 & \frac{1}{m-1}p_1 & \frac{1}{m-2} p_2 \\
\frac{1}{m-1}p_1 & 0 & \frac{1}{m-1} p_1 \\
\frac{1}{m-2} p_2 & \frac{1}{m-1}p_1 & 0 \\
& & & \ddots \\
& & & & 0
\end{array}\right]
\]
and hence for the Froebenius norm
\[
\| P_m-R_m\|^2 = \sum_{k=1}^n \frac{m-k}{(m-k)^2}p^2_k = \sum_{k=1}^n \frac{1}{m-k} p^2_k.
\]
\end{proof}
  
We are now ready to state our main result.

\begin{theorem}\label{main_result}
\[
\lim_{m\rightarrow +\infty}\lambda_{\min}(P_m)=\min_\theta p(\theta).
\]
\end{theorem}

\begin{proof}
Let $v$ be an eigenvector  of $P_m$ such that  $v^Tv=1$ and $P_mv=\lambda_{\min}(P_m)v$. From 
the equality
\[
v^TP_mv = v^T(P_m-R_m)v+v^TR_mv,
\]
we obtain with the help of Lemma \ref{lem2} the following inequality
\[
\lambda_{\min}(P_m)\geq O(m^{-\frac{1}{2}})+  \lambda_{\min}(R_m).
\]
Taking the limit, we obtain
\[
\liminf_{m\rightarrow +\infty}\lambda_{\min}(P_m)\geq\lim_{m\rightarrow +\infty}\lambda_{\min}(R_m).
\]
Using Lemma \ref{lem1} and Theorem \ref{Szego}, we can see that
\[
\liminf_{m\rightarrow +\infty}\lambda_{\min}(P_m)\geq \lim_{m\rightarrow +\infty}\lambda_{\min}(R_m)=\min_{\theta} p(\theta)
\]
and hence
\[
\liminf_{m\rightarrow +\infty}\lambda_{\min}(P_m)\geq \min_\theta p(\theta)\geq \limsup_{m\rightarrow +\infty}\lambda_{\min}(P_m).
\]
If $x_m$ is a real sequence then it is well-known that if $\liminf_{m\rightarrow +\infty}x_m=\limsup_{m\rightarrow +\infty}x_m$, then the sequence $x_m$ converges to
$\lim_{m\rightarrow +\infty}x_m=\liminf_{m\rightarrow +\infty}x_m=\limsup_{m\rightarrow +\infty}x_m$,
and this completes the proof.
\end{proof}

\begin{corollary} 
Assume that polynomial $p$ is positive.
Then, there exists a sufficiently large integer $m_0$ such that for $m \geq m_0$,
the Toeplitz matrix $P_m$ is positive definite.
\end{corollary}

\begin{proof}
Use Theorem \ref{main_result}.
\end{proof}

\begin{remark}\label{rem1}
Note that when $p$ is positive,
matrices $P_m$ are not necessarily positive definite if $m$ is not large enough.
As a simple example consider the positive polynomial
$p(\theta)=2+2\cos\theta+\frac{8}{5}\cos2\theta$.
We have
\[
P_3=\left[\begin{array}{ccc}
2 & \frac{3}{2} &  \frac{12}{5}\\[.5em]
\frac{3}{2}   & 2  &  \frac{3}{2} \\[.5em]
\frac{12}{5} & \frac{3}{2}  & 2 
\end{array}\right]
\]
which is not positive definite, since $\lambda_{\min}(P_3)= -\frac{2}{5}$.
Also, the next Toeplitz matrix
\[
P_4=\left[\begin{array}{cccc}
2 &  \frac{4}{3}   & \frac{8}{5}    &     0\\[.5em]
\frac{4}{3} &  2  &  \frac{4}{3}  &  \frac{8}{5}\\[.5em]
 \frac{8}{5} &   \frac{4}{3}   & 2 &   \frac{4}{3}\\[.5em]
 0  &   \frac{8}{5}  &  \frac{4}{3}  &  2
\end{array}\right],
\]
is not positive definite either, since $\lambda_{\min}(P_4)=\frac{8}{13}-\frac{2\sqrt{509}}{15} \approx -0.3415$.
However, one can check that when $m\geq m_0=30$, matrices
$P_m$ are indeed positive definite.
\end{remark}

\section{LMI inner approximations of stability domain}

Consider a monic polynomial
\[
d(z)=d_0+d_1z+\cdots+d_{n-1}z^{n-1}+z^n
\]
of degree $n$, with coefficient vector $d \in {\mathbb R}^n$
and let us define the set
\[
{\mathcal S} = \{d \in {\mathbb R}^n \: :\: d(z) \:\mathrm{stable} \}
\]
where stability is meant in the discrete-time, or Schur sense, i.e. all
the roots of $d(z)$ belong to the open unit disk. Many control problems
(e.g. fixed-order or robust controller design) can be formulated as
linear programming problems in $\mathcal S$. Unfortunately $\mathcal S$
is nonconvex when $n>2$, which renders controller design difficult
in general. It can therefore be relevant to describe convex inner
approximations of $\mathcal S$, in particular by exploiting the
modeling flexibility of linear matrix inequalities (LMIs),
see \cite{hsk03} and references therein.

An approach consists in choosing a monic polynomial
\[
c(z)=c_0+c_1z+\cdots+c_{n-1}z^{n-1}+z^n
\]
which is stable. Once $c$ is given, we define the trigonometric polynomial
\[
\begin{array}{rcl}
z=e^{i\theta} \mapsto p^{c,d}(\theta) & = &
c(z^{-1})d(z)+c(z)d(z^{-1})
\\
& = & 2\displaystyle\sum_{l=0}^n\sum_{\begin{subarray}{c}j,k=0\\
|j-k|=l\end{subarray}}^nc_jd_k\cos l\theta
\end{array}
\]
and the set
\[
{\mathcal P}^c = \{d \in {\mathbb R}^n \: :\: p^{c,d}(\theta) > 0
\quad\forall\:\theta \in {\mathbb R}\}.
\]

\begin{lemma}\label{schur}
Let $c(z)$ be a given stable polynomial. Then ${\mathcal P}^c \subset {\mathcal S}$.
\end{lemma}

\begin{proof}
A geometric proof is as follows. Since polynomial $c(z)$ is Schur stable, 
when $z=e^{i\theta}$ varies along the unit circle, 
complex number $c(e^{i\theta})$ has a net increase of argument of $2n\pi$,
or equivalently the plot of $c(e^{i\theta})$ encircles the origin $n$ times, see
e.g. \cite[Section 1.3.3]{bck} or use Cauchy's argument principle. Notice that
the real number $p^{c,d}(\theta)=c(e^{-i\theta})d(e^{i\theta})+c(e^{i\theta})d(e^{-i\theta})$
is equal to $2|c(e^{i\theta})d(e^{i\theta})|\cos(c(e^{i\theta}),d(e^{i\theta}))$
where the last term is the cosine of the oriented angle between vectors
$c(e^{i\theta})$ and $d(e^{i\theta})$ in the complex plane. Therefore
$p^{c,d}(\theta)$ positive implies that the cosine is positive and hence that
the angle between $c(e^{i\theta})$ and $d(e^{i\theta})$
is less than $\frac{\pi}{2}$ in absolute value for any given
value of $\theta$. This means
that complex number $d(e^{i\theta})$ also encircles the
origin $n$ times when $\theta$ range from $0$ to $2\pi$, and hence
that polynomial $d(z)$ is Schur stable.
\end{proof}

Let $P^{c,d}_m$ be the symmetric banded Toeplitz matrix
corresponding to polynomial $p^{c,d}$, built as in (\ref{pm}), and define the set
\[
{\mathcal P}^c_m = \{d \in {\mathbb R}^n \: :\: P^{c,d}_m \succ 0\}
\]
where $\succ 0$ means positive definite. Note that symmetric matrix $P^{c,d}_m$ depends
affinely on $d$, so that ${\mathcal P}^c_m$ is a convex LMI set.

\begin{theorem}\label{inner}
Let $c(z)$ be a given stable polynomial of degree $n$, and let $m>n$.
Then ${\mathcal P}^c_m \subset {\mathcal S}$.
\end{theorem}

\begin{proof}
Since $mp^{c,d}(\theta)=v^T_m(e^{-i\theta})P^{c,d}_mv_m(e^{i\theta})$,
positive definiteness of matrix $P^{c,d}_m$ implies positivity of
polynomial $p^{c,d}(\theta)$. Then use Lemma \ref{schur}.
\end{proof}

Set ${\mathcal P}^c_m$ is therefore a valid convex inner approximation
of the nonconvex stability region $\mathcal S$. Its geometry depends
only on the choice of a stable polynomial $c(z)$.

\begin{theorem}\label{limit}
Let $c(z)$ be a given stable polynomial. Then
${\mathcal P}^c = \lim_{m\rightarrow +\infty} {\mathcal P}^c_m$.
\end{theorem}

\begin{proof}
Use Theorem \ref{main_result}.
\end{proof}

Finally we make the connection with the results in \cite{hsk03}.
Recall that a discrete-time rational function is strictly
positive real (SPR) whenever its real part is strictly positive
when evaluated along the unit circle.

\begin{theorem}
\[
{\mathcal P}^c = \{d \in {\mathbb R}^n \: :\: \frac{d(z)}{c(z)}
\:\:\mathrm{SPR} \}.
\]
\end{theorem}

\begin{proof}
Since $c(z)$ is stable, the SPR inequality
\[
\mathrm{Re}\:\frac{d(e^{i\theta})}{c(e^{i\theta})} = 
\frac{1}{2}\left(\frac{d(e^{i\theta})}{c(e^{i\theta})}+
\frac{d(e^{-i\theta})}{c(e^{-i\theta})}\right) = 
\frac{c(e^{-i\theta})d(e^{i\theta})+c(e^{i\theta})d(e^{-i\theta})}
{2|c(e^{i\theta})|^2} > 0
\]
is equivalent to positivity of trigonometric polynomial $p^{c,d}(\theta)$.
\end{proof}

Polynomial $c(z)$ is referred to as a central polynomial in \cite{hsk03}
since set ${\mathcal P}^c$ is built around $c(z)$ in the coefficient
space. Note however that there is no guarantee that $c(z)$ belongs
to ${\mathcal P}^c_m$ if $m$ is not large enough, see Remark \ref{rem1}.

\section{Example}

\subsection{Second-order polynomials}

\begin{figure}[h!]
\begin{center}
\includegraphics[scale=0.8]{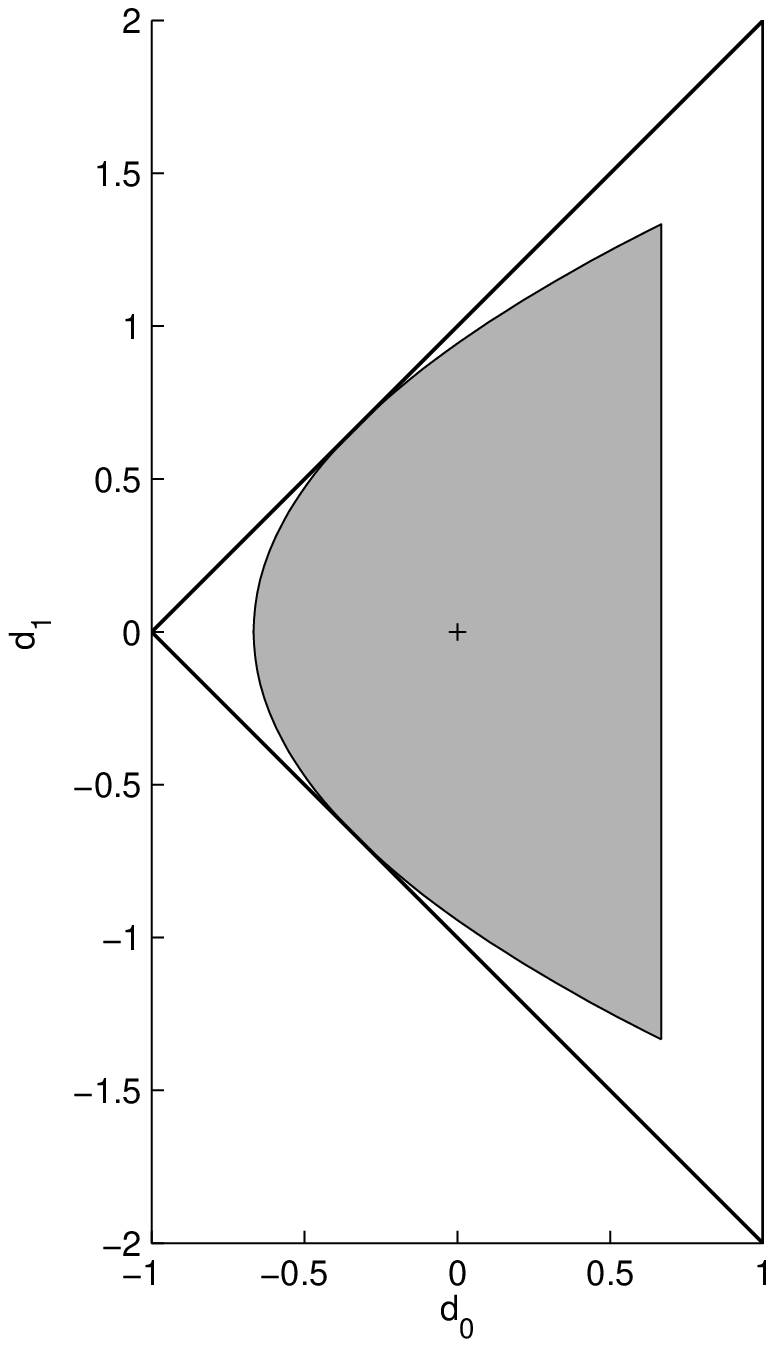}
\includegraphics[scale=0.8]{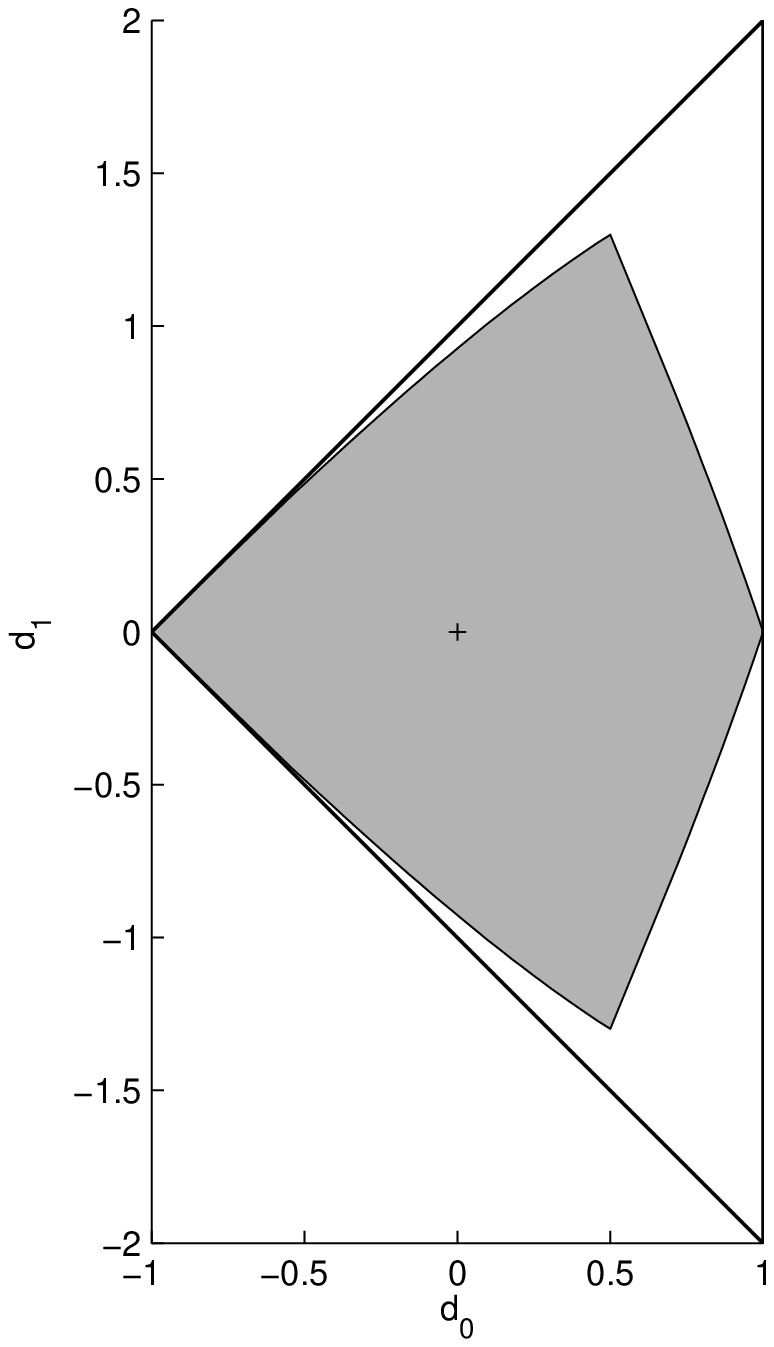}
\caption{3-by-3 LMI set (shaded gray, left)
and 4-by-4 LMI set (shaded gray, right) within second-order discrete-time
stability region (triangle).\label{lmi34}}
\end{center}
\end{figure}

We consider second-order polynomials for which the exact stability region
is a triangle with vertices $(z+1)^2$, $(z-1)(z-1)$ and $(z-1)^2$
\cite[Example 11.13]{ackermann}.

Choosing $c(z)=z^2$, we have $p^{c,d}(\theta)=2+2d_1\cos\theta+2d_0\cos2\theta$.
The first LMI inner approximation is
\[
{\mathcal P}^{c,d}_3 = \{(d_0,d_1) \: :\: P^{c,d}_3 = \left[\begin{array}{ccc}
2 & \frac{3}{2}d_1 & 3d_0 \\
\frac{3}{2}d_1 & 2 & \frac{3}{2}d_1 \\
3d_0 & \frac{3}{2}d_1 & 2
\end{array}\right] \succ 0\}
\]
and it is represented on the left of Figure \ref{lmi34} within the stability
triangle, as claimed by Theorem \ref{inner}.

The second LMI inner approximation is
\[
{\mathcal P}^{c,d}_4 = \{(d_0,d_1) \: :\: P^{c,d}_4 = \left[\begin{array}{cccc}
2 & \frac{4}{3}d_1 & 2d_0 & 0\\
\frac{4}{3}d_1 & 2 & \frac{4}{3}d_1 & 2d_0\\
2d_0 & \frac{4}{3}d_1 & 2 & \frac{4}{3}d_1 \\
0 & 2d_0 & \frac{4}{3}d_1 & 2
\end{array}\right] \succ 0\},
\]
see the right of Figure \ref{lmi34}.

\begin{figure}[h!]
\begin{center}
\includegraphics[scale=0.8]{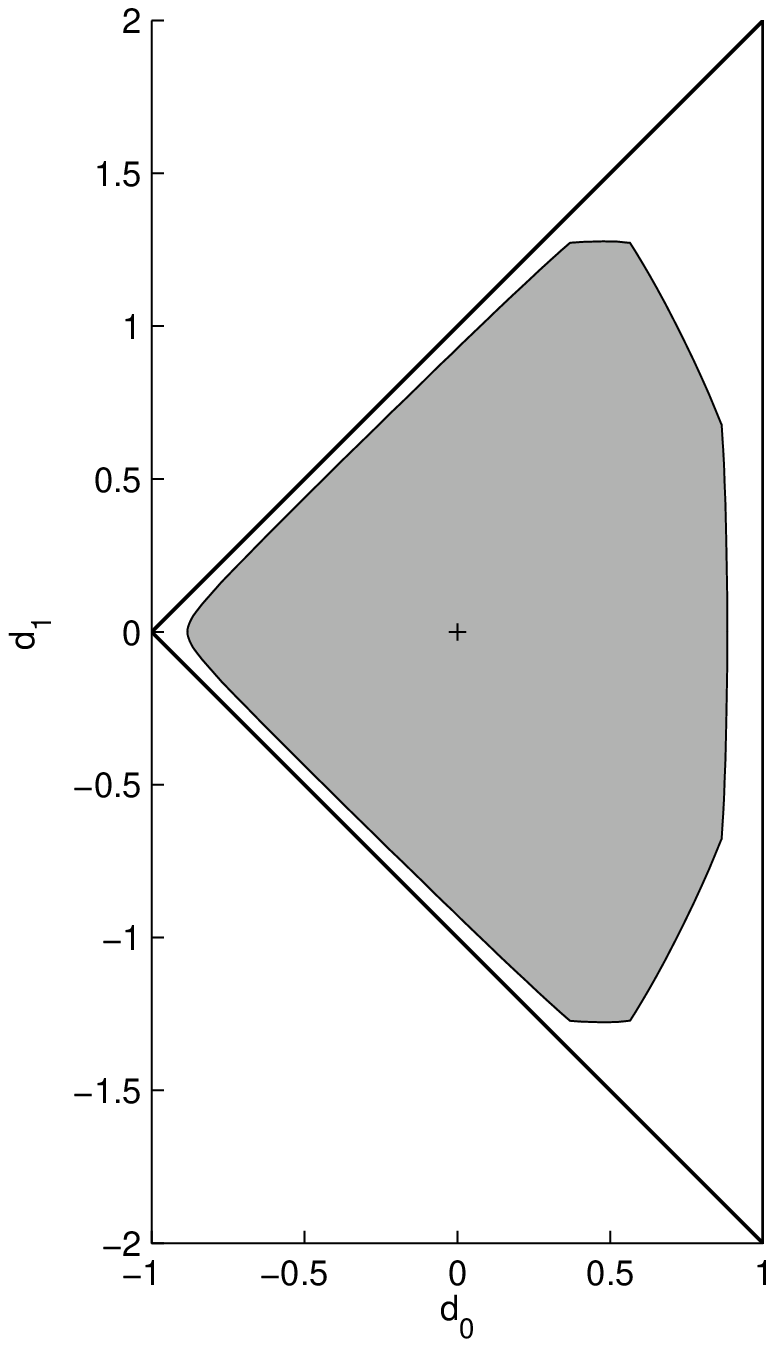}
\includegraphics[scale=0.8]{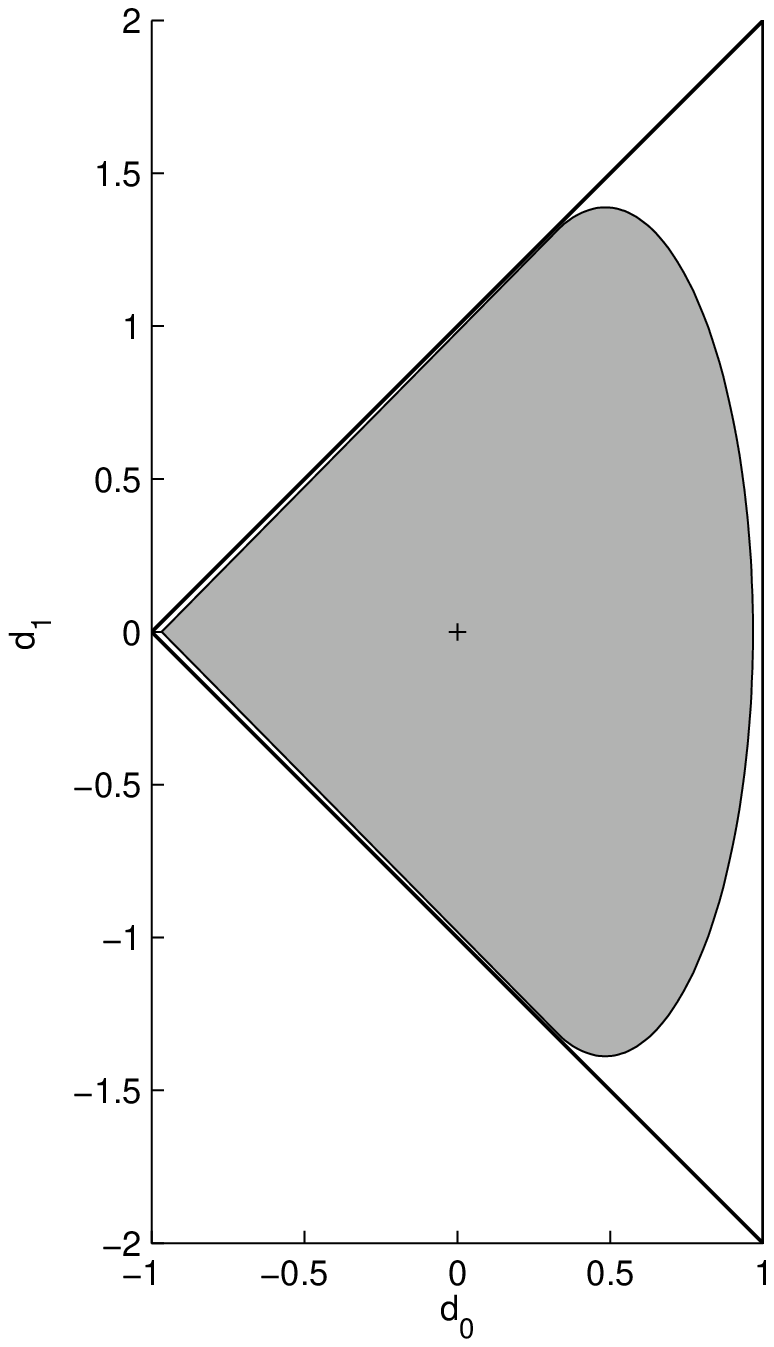}
\caption{7-by-7 LMI set (shaded gray, left)
and 50-by-50 LMI set (shaded gray, right) within second-order discrete-time
stability region (triangle).\label{lmi750}}
\end{center}
\end{figure}

On the left of Figure \ref{lmi750}
we represent the LMI set
\[
{\mathcal P}^{c,d}_7 = \{(d_0,d_1) \: :\: P^{c,d}_7 = \left[\begin{array}{ccccccc}
2 & \frac{7}{6}d_1 & \frac{7}{5}d_0 & 0 & 0 & 0 & 0\\
\frac{7}{6}d_1 & 2 & \frac{7}{6}d_1 & \frac{7}{5}d_0 & 0 & 0 & 0\\
\frac{7}{5}d_0 & \frac{7}{6}d_1 & 2 & \frac{7}{6}d_1 & \frac{7}{5}d_0 & 0 & 0 \\
0 & \frac{7}{5}d_0 & \frac{7}{6}d_1 & 2 & \frac{7}{6}d_1 & \frac{7}{5}d_0 & 0 \\
0 & 0 & \frac{7}{5}d_0 & \frac{7}{6}d_1 & 2 & \frac{7}{6}d_1 & \frac{7}{5}d_0 \\
0 & 0 & 0 & \frac{7}{5}d_0 & \frac{7}{6}d_1 & 2 & \frac{7}{6}d_1 \\
0 & 0 & 0 & 0 & \frac{7}{5}d_0 & \frac{7}{6}d_1 & 2
\end{array}\right] \succ 0\}.
\]
The boundary of this set
is piecewise polynomial, defined by two algebraic plane
curves whose irreducible defining polynomials $-7200 + 5040d_0 + 3528d_0^2
+ 4900d_1^2 - 5145d_0d_1^2$ (a cubic) and $6480000 + 4536000d_0 - 9525600d_0^2
- 8820000d_1^2 - 4445280d_0^3 + 7717500d_0d_1^2 + 3111696d_0^4 
- 4321800d_0^2d_1^2 + 1500625d_1^4$ (a quartic)
factor the determinant of the 7-by-7 pencil $P^{c,d}_7$.
See Figure \ref{lmi7closure} for a representation of
this set and the algebraic components of its boundary.

\begin{figure}[h!]
\begin{center}
\includegraphics[scale=0.8]{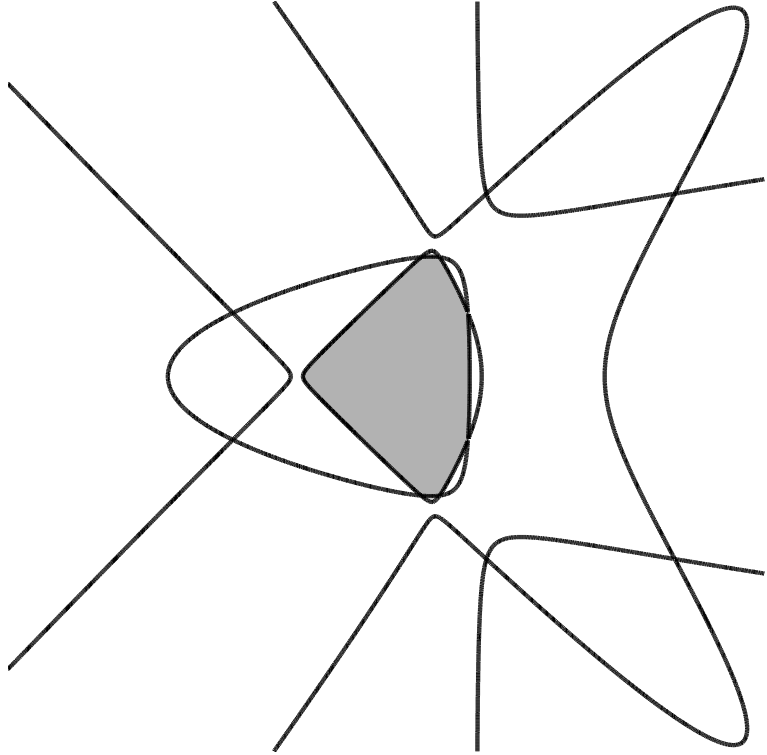}
\caption{7-by-7 LMI set (shaded gray)
and the algebraic components of its boundary (thick black lines).\label{lmi7closure}}
\end{center}
\end{figure}

On the right of Figure \ref{lmi750}
we represent the LMI set ${\mathcal P}^{c,d}_{50}$
which, according to Theorem \ref{limit},
is almost equal to the lifted LMI set
\[
{\mathcal P}^{c,d} = 
\{(d_0,d_1) \: :\:\exists\:(q_0,q_1,q_2) \: :\:
\left[\begin{array}{ccc}
q_0 & q_1 & d_0 \\
q_1 & q_2-q_0 & d_1-q_1 \\
d_0 & d_1-q_1 & 2-q_2
\end{array}\right] \succ 0\},
\]
the projection onto ${\mathbb R}^2$ of an LMI living
in ${\mathbb R}^5$, and which is the union of an ellipse and a triangle,
as studied in \cite{hsk03}.

\subsection{Third order}

\begin{figure}[h!]
\begin{center}
\includegraphics[scale=0.6]{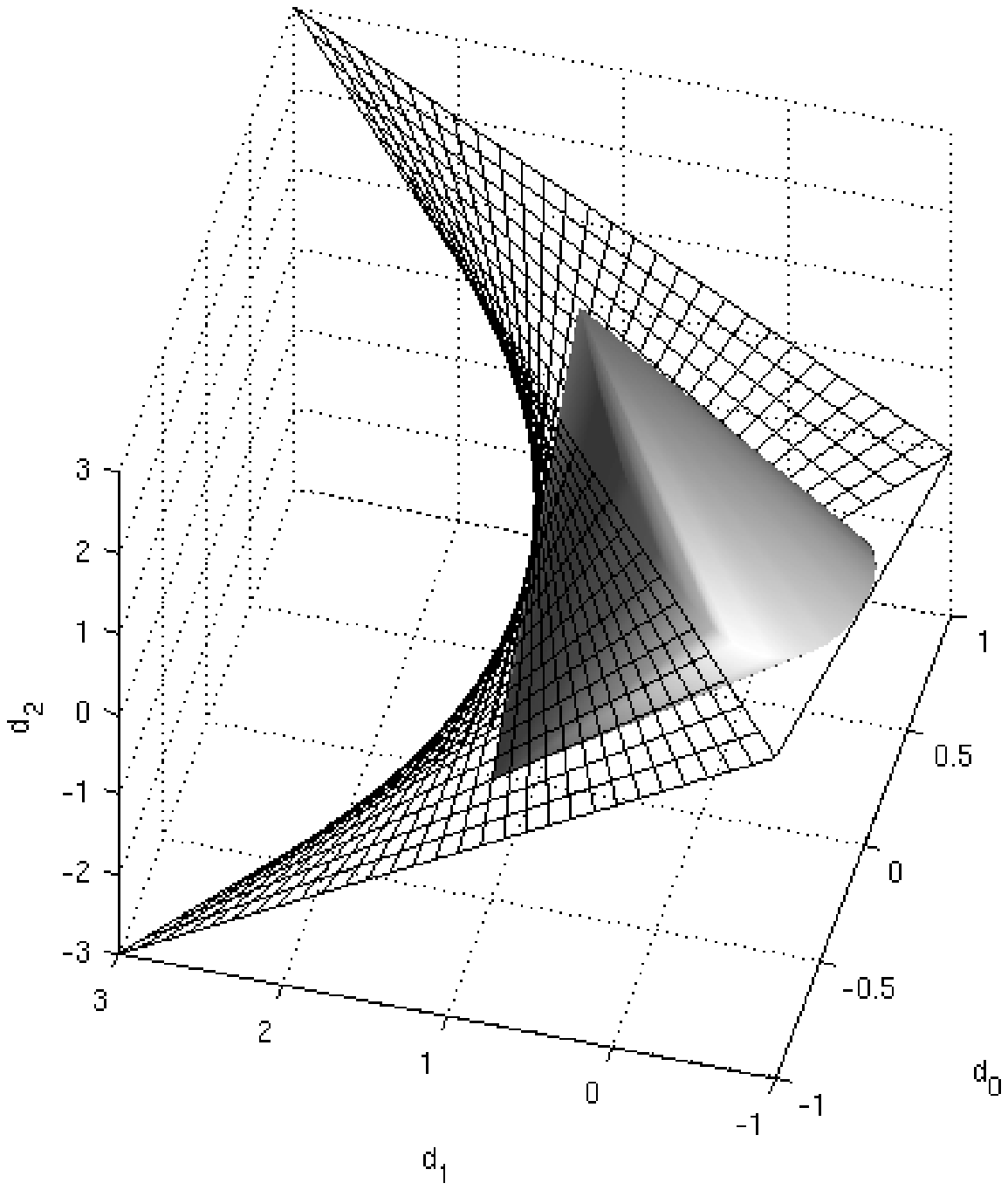}
\includegraphics[scale=0.6]{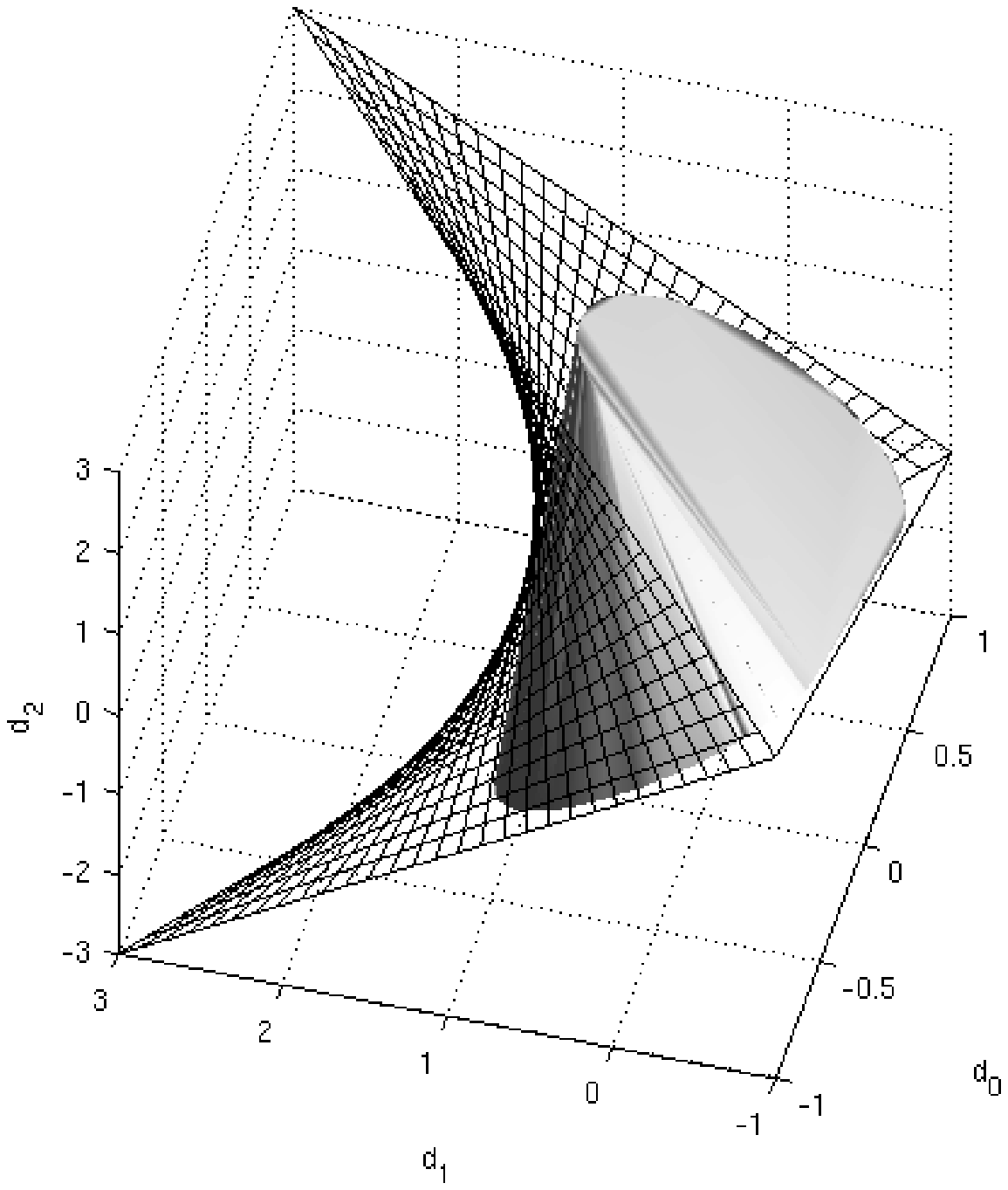}
\caption{4-by-4 LMI set (shaded gray, left)
and 50-by-50 LMI set (shaded gray, right) within third-order discrete-time
stability region (delimited by a meshed hyperbolic parabolic embedded in a
tetrahedron).\label{spacelmi}}
\end{center}
\end{figure}

We consider third-order polynomials for which the exact stability region
is delimited by a nonconvex hyperbolic parabolic embedded in a tetrahedron
with vertices $(z+1)^3$, $(z+1)^2(z-1)$, $(z+1)(z-1)^2$ and $(z-1)^3$,
see \cite[Example 11.14]{ackermann}.

Choosing $c(z)=z^3$, we have $p^{c,d}(\theta)=2+2d_2\cos\theta+2d_1\cos2\theta+2d_0\cos3\theta$.
The first LMI inner approximation is
\[
{\mathcal P}^{c,d}_4 = \{(d_0,d_1,d_2) \: :\: P^{c,d}_4 = \left[\begin{array}{cccc}
2 & \frac{4}{3}d_2 & 2d_1 & 4d_0 \\
\frac{4}{3}d_2 & 2 & \frac{4}{3}d_2 & 2d_1 \\
2d_1 & \frac{4}{3}d_2 & 2 & \frac{4}{3}d_2 \\
4d_0 & 2d_1 & \frac{4}{3}d_2 & 2
\end{array}\right] \succ 0\}
\]
and it is represented on the left of Figure \ref{spacelmi} within the nonconvex stability
region, as claimed by Theorem \ref{inner}. On the right of Figure \ref{spacelmi}
is represented the LMI set ${\mathcal P}^{c,d}_{50}$
which, according to Theorem \ref{limit}, is almost equal to the lifted LMI set
${\mathcal P}^{c,d}$.

\section{Concluding Remarks}

We have used results on spectra of Toeplitz matrices to construct a hierarchy
of convex inner approximations of the nonconvex set of stable polynomials,
with potential applications in fixed-order robust controller design.
The main difference with respect to previous results is that the
inner sets are defined by LMIs (affine sections of the
cone of positive definite matrices) without the need to resort
to projections and lifting variables. Moreover, our LMI sets
belong to a hierarchy converging asymptotically to a lifted LMI
inner approximation described previously in \cite{hsk03}.

It is likely that our results can be extended to deal with
positive trigonometric polynomial matrices and block Toeplitz matrices,
with potential applications in multi-input multi-output control systems.

Sufficient conditions ensuring that a real polynomial is a sum-of-squares
(and hence that it is positive) have been proposed in \cite{lasserre},
so it could be insightful to transpose these conditions to trigonometric
polynomials and compare with our approach. Results in \cite{lasserre}
are also valid for multivariate polynomials, and this may have applications
in fixed-order or robust controller design for multi-dimensional systems.

\section*{Acknowledgements}

The first author is supported by a Ram\'on y Cajal grant from the
Spanish government. The second author acknowledges support by 
project No.~103/10/0628 of the Grant Agency of the Czech Republic.


\begin{thebibliography}{XXX}

\bibitem{ackermann}
J. Ackermann et al.
Robust control: the parameter space approach.
2nd edition, Springer, 2002.

\bibitem{bck}
S. P. Bhattacharyya, H. Chapellat, L. H. Keel.
Robust control: the parametric approach.
Prentice Hall, 1995.

\bibitem{dumitrescu}
B. Dumitrescu. Positive trigonometric polynomials and signal processing applications.
Springer, 2007.

\bibitem{karimi}
A. Karimi, H. Khatibi, R. Longchamp. Robust control of polytopic systems by convex optimization.
Automatica 43(6):1395-1402, 2007.

\bibitem{gray}
R. M. Gray. Toeplitz and circulant matrices: a review.
Foundations and Trends in Communications and Information Theory 2(3):155-239, 2006.
       
\bibitem{hsk03}
D. Henrion, M. \v Sebek, V. Ku\v cera. Positive polynomials and robust stabilization
with fixed-order controllers. IEEE Trans. Autom. Control 48(7):1178-1186, 2003.

\bibitem{csm}
D. Henrion, J. B. Lasserre. Solving nonconvex optimization problems - How GloptiPoly
is applied to problems in robust and nonlinear control. IEEE Control Syst. Mag.
24(3):72-83, 2004.

\bibitem{lasserre}
J. B. Lasserre. Sufficient conditions for a polynomial to be a sum of squares.
Arch. Math. 89:390--398, 2007.

\end{thebibliography}
\end{document}